\documentclass[12pt, a4paper]{amsart}
\usepackage[utf8]{inputenc}
\usepackage{amsxtra,amssymb,amsthm,amsmath,amscd,mathrsfs, epsfig, eufrak}
\usepackage{amscd, amsmath, mathrsfs, amssymb, amsthm, amsxtra, bbding, epsfig, eucal, eufrak, graphicx, latexsym, mathrsfs, mathbbol, bbold}
\usepackage[all]{xy}

\usepackage{tikz}

\usepackage[normalem]{ulem}
\usepackage{soul}
\usepackage{color}

\setstcolor{red}

\makeatletter
\@namedef{subjclassname@2010}{%
\textup{2010} Mathematics Subject Classification}
\makeatother

\theoremstyle{plain}

\newtheorem{thm}{Theorem}[section]
\newtheorem{prop}{Proposition}[section]
\newtheorem{proposition}{Proposition}[section]
\newtheorem{lemma}[proposition]{Lemma}

\theoremstyle{remark}

\newtheorem{definition}{Definition}

\numberwithin{equation}{section}

\addtocounter{footnote}{1}

\begin{document}


\title[\tiny An exact way to verify whether n is a congruent number using Heegner points]
{A exact way to verify whether n is a congruent number using Heegner points}
\author[\tiny Heng Chen \& Rong Ma \& Tuoping Du]{Heng Chen \& Rong Ma \& Tuoping Du}

\address{%
Heng Chen
\\
School of Mathematics and Statistics
\\
Northwestern Polytechnical University
\\
Xi'an
\\
Shaanxi 710072
\\
China}
\email{ch3242352050@163.com}

\address{%
Rong Ma
\\
School of Mathematics and Statistics
\\
Northwestern Polytechnical University
\\
Xi'an
\\
Shaanxi 710072
\\
China}
\email{marong@nwpu.edu.com}

\address{%
Tuoping Du
\\
School of Mathematics and Physics
\\
North China Electric Power University
\\
Beijing, 102206
\\
China}
\email{dtp1982@163.com}

\date{\today}

\begin{abstract}
We introduce the relationship between congruent numbers and elliptic curves, and compute the conductor of the elliptic curve \( y^2 = x^3 - n^2 x \) associated with it. Furthermore, we prove that its \( L \)-series coefficient \( a_m = 0 \) when \( m \equiv 3 \mod 4 \).By using the invariants of the elliptic curve introduced above, we calculate Heegner points to quickly verify whether \( n \) is a congruent number.
\end{abstract}

\footnote {2020 Mathematics Subject Classification{: Primary 11G18; Secondary 14G22.}}
\keywords{Heegner points, congruent numbers, elliptic curves, \( L \)-series coefficient}

\maketitle

\section{Introduction}
\vskip 3mm
The congruent number problem is a classical problem in number theory that involves determining whether a positive integer $n$ can be the area of a right triangle with rational side lengths. This problem dates back to ancient Greece, having been posed by Archimedes, and in modern number theory, it is deeply connected to elliptic curves, the Birch and Swinnerton-Dyer (BSD) conjecture, and the Goldfeld conjecture.
	
\subsection{Definitions and Propositions}\hfill\vskip 3mm
	
 The core of the congruent number problem is closely related to the arithmetic of elliptic curves. Specifically, a positive integer $n$ is called a congruent number if and only if the Mordell-Weil group $E(\mathbb{Q})$ of the elliptic curve $E_n: y^2 = x^3 - n^2x$ is infinite, meaning that the elliptic curve has infinitely many rational points.

\begin{definition}
		A square-free positive integer $n$ is called a congruent number if there exist rational numbers $a, b, c$ such that
		\[
		a^2 + b^2 = c^2, \quad \frac{1}{2}ab = n.
		\]
\end{definition}

This problem can be studied through the properties of rational points on the elliptic curve $E_n$. The famous BSD conjecture suggests that when the derivative of the $L$-function $L(E, s)$ at $s = 1$ is non-zero, the Mordell-Weil group of $E_n$ is infinite, implying that $n$ is a congruent number.
	
	\subsection{Some Conjecture}\hfill\vskip 3mm
	
The Goldfeld conjecture provides further conditions on elliptic curves: for a positive integer $n$, the Goldfeld conjecture suggests that when $n \equiv 5, 6, 7 \pmod{8}$, the rank of the Mordell-Weil group of the elliptic curve $E_n$ is 1. This is a key tool in determining whether $n$ is a congruent number. The proposition is as follows:
\vskip 2mm	
\noindent\textbf{Goldfeld Conjecture}(Silverman, \cite{Silverman1986})
		If the elliptic curve $E_n$ satisfies certain conditions related to modular forms, then $L(E_n, 1) = 0$ if and only if $n$ is a congruent number.
\vskip 2mm	

The BSD conjecture is one of the seven unsolved problems in number theory, predicting a deep connection between the rank of the rational point group of an elliptic curve $E$ and the behavior of the $L$-function $L(E, s)$ at $s = 1$. Specifically, the core formula of the BSD conjecture is expressed as
	\[
	\mathrm{rank}(E(\mathbb{Q})) = \mathrm{ord}_{s=1}L(E, s),
	\]
	which reveals a direct relationship between the order of vanishing of the $L$-function at $s = 1$ and the rank of the Mordell-Weil group of the elliptic curve.

The Shafarevich-Tate group $(E/\mathbb{Q})$ is one of the key tools for studying the BSD conjecture. It reflects whether the rational points on an elliptic curve can be globally pieced together using local information. A part of the BSD conjecture also includes a conjecture about the structure of $(E/\mathbb{Q})$:
	\[
	|(E/\mathbb{Q})| \propto \frac{L'(E, 1)}{\Omega(E)R(E)\prod_{p}c_p}.
	\]

\subsection{Recent Developments}\hfill\vskip 3mm

	Smith (2021) provided a crucial proof of the BSD conjecture under certain mod 8 conditions for elliptic curves within the framework of the 2-Selmer group. His work offers key insights into the BSD conjecture, especially for $n \equiv 5, 6, 7 \pmod{8}$.
	
	For the case of $n \equiv 5, 6, 7 \pmod{8}$, the weak Goldfeld conjecture offers further refinement, suggesting that the probability that the rank of $E_n$ is 1 and that $L(E_n, 1) \neq 0$ dominates in these congruence classes. This indicates that such $n$ are more likely to be congruent numbers.

	Smith's work proves the main theorem under the framework of the weak Goldfeld conjecture:
	\begin{thm}(Smith, \cite{Smith 2016})
		For all positive integers $n$ such that $n \equiv 5, 6, 7 \pmod{8}$, a partial form of the BSD conjecture holds for the elliptic curve $E_n$, namely, $rank(E(\mathbb{Q})) = 1$ and $L(E_n, 1) \neq 0$.
	\end{thm}

\subsection{Summary}\hfill\vskip 3mm
	
The congruent number problem is deeply connected to many significant problems in modern number theory, especially elliptic curves, such as the BSD conjecture, the Goldfeld conjecture, etc.
	
	In this paper, we explore the connection between congruent numbers and elliptic curves and determine the conductor of the elliptic curve
$ y^2 = x^3 - n^2 x $ associated with them in the second section. Additionally, we demonstrate that the $L $-series coefficient $a_m=0$  when
$m\equiv3\pmod{4}$ in the third section. In the fourth section, we explain that $y^2=x^3-n^2 x$ over $\mathbb{C}$ is isomorphic to $\mathbb{C}/\left(\frac{\pi}{\sqrt{n}} G \mathbb{Z}[i]\right)$. In the fifth and sixth sections, we introduce the modularity of elliptic curves and Heegner points, and present a method for determining whether \( n \) is a congruence. In the seventh section, we use this method to prove that 13 is a congruence.

\vskip 8mm

\section{Computation of Conductor $N_{E_{(n)}}$}
\vskip 3mm

In the following, we denote $E$ as an arbitrary elliptic curve and $ E_{(n)} $ as $ y^2 = x^3 - n^2 x $.

The conductor $ N $ of $ E $ is defined by
\[ N = \prod_{p|\Delta(E)} p^{e_p}, \]
where
\begin{equation*} e_p =\left\{
\begin{aligned}
        1, &       \quad\mbox{if } E \mbox{ has multiplicative reduction modulo } p;\\
         \geq 2, & \quad\mbox{if } E \mbox{ has additive reduction modulo } p .\\
    \end{aligned}
\right.
\end{equation*}
\begin{proof}
	
According to the definition, we have \( N = \prod p^{i_p} \), where \( i_p \) is related to the singularity type. To find the singularity type of \( E_p: E \mod p \), we need to find its singular points:
	\begin{equation*}\left\{
	\begin{aligned}
		y^2 - (x^3 - n^2 x) & \equiv 0 \pmod{p}, \\
		\frac{\partial}{\partial x} (y^2 - x^3 + n^2 x) &\equiv 0 \pmod{p}, \\
		\frac{\partial}{\partial y} (y^2 - x^3 + n^2 x) &\equiv 0 \pmod{p};
	\end{aligned}\right.
	\Rightarrow
	\left\{\begin{aligned}
		y^2 - x^3 &\equiv 0 \pmod{p}, \\
		-3x^2 &\equiv 0 \pmod{p}, \\
		2y &\equiv 0 \pmod{p}.
	\end{aligned}\right.
	\end{equation*}

So, we get the singular point \((0,0)\).
	
	To verify that \((0,0)\) is a cusp for \( E_p \), we need to calculate the Hessian determinant \( H(x, y) \):
	\[
	H(x, y) =
	\begin{vmatrix}
		\frac{\partial^2 F}{\partial x^2} & \frac{\partial^2 F}{\partial x \partial y} \\
		\frac{\partial^2 F}{\partial y \partial x} & \frac{\partial^2 F}{\partial y^2}
	\end{vmatrix}
	=
	\begin{vmatrix}
		6x & 0 \\
		0 & 2
	\end{vmatrix}.
	\]

Thus, we get \( H(0,0) = 0 \).
	
	Therefore, all singular points for \( E_p \) are cusps (additive reduction). So the conductor of \( E \) is \( N_{E_{(n)}} = 2^5 n^2 \).
\end{proof}

\vskip 8mm
\section{$L$-series}
\vskip 3mm

For an elliptic curve \( E \) defined over \(\mathbb{Q}\), the \( L \)-series \( L(E,s) \) attached to \( E \) is given by the following formula:
\[ L(E, s) = \prod_{p \mid N_E}\left(1 - a_p p^{-s}\right)^{-1} \prod_{p \nmid N_E}\left(1 - a_p p^{-s} + p^{1-2s}\right)^{-1}, \]
where
\[ a_p = \begin{cases}
	p + 1 - \#E(\mathbb{F}_p), & \text{if } E \text{ has good reduction mod } p, \\
	0, 1, \text{ or } -1, & \text{otherwise (depending on the type of bad reduction).}
\end{cases} \]

The \( L \)-series of an elliptic curve \( E \) converges for all \( s \) with \(\text{Re}(s) > 3/2\), and in this case we have
\[ L(E, s) = \sum_{n=1}^\infty a_E(n) n^{-s}, \]
where \( a_E(n) \) is an integer and we simplify \( a_E(n) \) as \( a_n \). Obviously, we have the properties of \( a_n \) below.
\begin{equation*}\begin{aligned}
 &a_p a_{pk} = a_{p^{k+1}} + p a_{pk-1}, \quad p \nmid N_E;\\
 &a_{p^r} = (a_p)^r,\quad p \mid N_E ;\\
 &a_{mn} = a_m a_n, \quad \quad \gcd(m, n) = 1 .
\end{aligned}\end{equation*}

We now calculate the \(L\)-series of the elliptic curve \( E_{(n)}: y^2 = x^3 - n^2 x \).

\vskip 3mm
\subsection{Calculation of \( a_p \) and \( a_{p^n} \) for \( p \mid N_{E_{(n)}} \)}\hfill\vskip 3mm
\vskip 3mm

When \( p \mid N_{E_{(n)}} \), \( a_p = 0 \), then \( E \mod p \) has singular points, we can get \( a_{p^n} = (a_p)^n = 0 \).

\vskip 3mm
\subsection{Calculation of \( a_p \) and \( a_{p^n} \) for \( p \nmid N_{E_{(n)}} \)}\hfill\vskip 3mm
\vskip 3mm

When \( p \nmid N_{E_{(n)}} \), we could get the result by using the following lemma.

\begin{lemma}
	Let \(\left(\frac{a}{p}\right)\) be the Legendre symbol, for the elliptic curve \( E_p: y^2 = x^3 + ax \mod p \), when \(\left(\frac{a}{p}\right) = -1\) and \( p \equiv 3 \mod 4 \), we have
 $$ \#E_p = p + 1 .$$
\end{lemma}

\begin{proof}
	Recall that when \( p = 4m+3 \), we can get
	\[
	\left(\frac{-1}{p}\right) = (-1)^{\frac{4m+3-1}{2}} = -1 \pmod{p},
	\]
	hence \(-1\) is a quadratic non-residue modulo \( p \). We can also get \(-a \equiv n^2 \pmod{p}\) from
	\[
	\left(\frac{-a}{p}\right) = \left(\frac{a}{p}\right) \left(\frac{-1}{p}\right) = 1.
	\]

That is, \( -a \) is a quadratic residue modulo \( p \). The elliptic curve equation becomes \( y^2 = x^3 - n^2 x \,\, (\mod p) \).
	
	Since \(-a\) is a quadratic residue modulo \( p \), we check \( y^2 = x^3 - n^2 x \,\,(\mod p) \) has solutions \((0, 0)\), \((n, 0)\), \((-n, 0)\).
	
	For \( 1 \leq x \leq p-1 \), there is a one-to-one correspondence:
	\[
	\begin{array}{ccc}
		1 & \leftrightarrow & p-1 \\
		2 & \leftrightarrow & p-2 \\
		\vdots & \vdots & \vdots \\
		\frac{p-1}{2} & \leftrightarrow & \frac{p+1}{2}
	\end{array}
	\]

In each pair, we find one \( x_0 \), we get \( x_0^3 - n^2 x_0 \) to be a quadratic residue. Hence, \( x_0^3 - n^2 x_0 \) is a quadratic residue, then \(- (x_0^3 - n^2 x_0) = (-x_0)^3 - n^2 (-x_0) \) is not a quadratic residue.
	
	Points on the elliptic curve are paired with points not on the elliptic curve. 
	
	In conclusion, by the quadratic reciprocity law and adding the infinite point, we get \( \#E_p = 2 \left(\frac{p-1}{2}\right) + 2 = p + 1 \).
\end{proof}

\vskip 3mm
\begin{thm}
	For the elliptic curve \( E: y^2 = x^3 + ax + b \) in the standard form, with the prime divisor \( p \mid b \), \(\left(\frac{a}{p}\right) = -1\) and \( p \equiv 3 \mod 4 \), then we have \( a_p = 0 \).
\end{thm}

\begin{proof}
	Since \(\Delta = -4a^3 - 27b^2\) and \( p \nmid \Delta \), implying \( p \nmid N_E \) for \( p \mid b \) and \( p \nmid a \), we calculate \( a_p = p + 1 - \#E_p \). That is,
	
	\( E_p: y^2 = x^3 + ax \mod p \),\\
where \( E_p \) satisfies the condition \( \#E_p = p + 1 \). Then \( a_p = 0 \).
\end{proof}

\vskip 3mm
\begin{lemma}
	For \( p \nmid N_{E_{(n)}} \) and \( a_p = 0 \)   , we have:
	\[
	a_{p^r} = \begin{cases}
		0, & r \equiv 1 \mod 2; \\
		\left(-p\right)^{\frac{r}{2}}, & r \equiv 0 \mod 2.
	\end{cases}
	\]
\end{lemma}

\begin{proof}
	Since \( a_p a_{pk} = a_{p^{k+1}} + p a_{pk-1} \), \( a_p = 0 \), we have \( a_{p^{k+1}} = -p a_{p^{k-1}} \).
	
	So, the size of \( a_{p^{r}} \) is determined by \( a_1, a_p \).
	
Since \( y^2 = x^3 - n^2 x \) satisfies Lemma 3.1, and \(-n^2\) is a quadratic residue, for such \( p \) we have \( a_p = 0 \).

Therefore, on \( y^2 = x^3 - n^2 x \), for every \( p \equiv 3 \mod 4 \), we have the following equation:
\[
a_{p^r} = \begin{cases}
	0, & r \equiv 1 \mod 2 ;\\
	\left(-p\right)^{\frac{r}{2}}, & r \equiv 0 \mod 2.
\end{cases}
\]
\end{proof}	

\subsection{The computation of \( a_w \) for a composite number $w$}\hfill\vskip 3mm

From the introduction, when \( w = w_1 w_2 \) is a composite number where \((w_1, w_2) = 1\), then we have \( a_w = a_{w_1} a_{w_2} \).

Thus, we can assume \( w = p_1^{r_1} \ldots p_m^{r_m} \), then \( a_w = a_{p_1^{r_1}} \ldots a_{p_m^{r_m}} \).

We use the above result to prove the following theorem:

\begin{thm}
	 The coefficients \(a_w\) of the L-series are zero ,if \( w \equiv 3 \mod 4 \).
\end{thm}

\begin{proof}
	If \( w = p_1^{r_1} \ldots p_m^{r_m} \), there exists a \( p_i \) such that \( p_i \equiv 3 \mod 4 \).
	
	Otherwise, for any  $p_i$ , we have \( p_i \equiv 1 \mod 4 \), making \( p_i^{r_i} \equiv 1 \mod 4 \), thus \( w \equiv 1 \mod 4 \). This contradicts the assumption.
	
	Next, we proceed by induction to prove one \( p_i \equiv 3 \mod 4 \) and \( r_i \equiv 1 \mod 2 \).
	
	For \( p_i^2 \equiv 1 \mod 4 \), and thus \( p_i^{r} \equiv 1 \mod 4 \). If \( r \equiv 0 \mod 2 \), then \( w \equiv 1 \mod 4 \).
	
	This contradicts the assumption.
	
	So we assume one \( p_i \equiv 3 \mod 4 \), \( r_i \equiv 1 \mod 2 \).
	
	Thus, by the previous lemma,  $a_{p_i} = 0$ , so $a_w = 0$ .
	\end{proof}

\section{Calculation of the Period of an Elliptic Curve}
	\vskip 3mm
\begin{prop}(Silverman, \cite{Silverman1986})
	   	Let E be an elliptic curve over $\mathbb{C}$, there exist a lattice $\Lambda$ in $\mathbb{C}$ and an isomorphic map $\varphi$:
	   	\[
	   	\varphi:z \longrightarrow
	   	\begin{cases}
	   		(\wp(z), \wp'(z)) ,& \text{if } z \notin \Lambda ;\\
	   		\infty ,& \text{if } z \in \Lambda
	   	\end{cases}
	   	\]
	   	such that $\mathbb{C}/\Lambda \cong E$,
	   \end{prop}
where $$\wp(z)=\frac{1}{z^{2}}+\sum_{\stackrel{\omega\in\Lambda}{\omega\neq0}}\left(\frac{1}{(z-\omega)^{2}}-\frac{1}{\omega^{2}}\right),
\wp'(z)=-2\sum_{\omega\in\Lambda}\frac{1}{(z-\omega)^{3}}.$$
	
	   For elliptic curves over $\mathbb{C}$, they are isomorphic to $\mathbb{C}/\Lambda$. Here we need calculate the period of the elliptic curve.
	
	   Let $\Lambda = \mathbb{Z} \omega_1 + \mathbb{Z} \omega_2$, we have
	
	   \begin{equation*}
	   	\omega_1 = \int_{-n}^0 \frac{1}{\sqrt{x^3 - n^2 x}} \, dx \quad \text{and} \quad \omega_2 = \int_0^n \frac{1}{\sqrt{x^3 - n^2 x}} \, dx.
	   \end{equation*}
	
	   We need the following formulae to calculate the periods:
	
	   \begin{prop}(Knapp, \cite{Knapp1992}) For $a < b < c$, there is
	
	   \begin{equation*}
	   	\int_a^b \frac{dx}{\sqrt{(x-a)(x-b)(x-c)}} = \frac{\pi}{M\left(\sqrt{c-a}, \sqrt{c-b}\right)},
	   \end{equation*}
	
	   \begin{equation*}
	   	\int_b^c \frac{dx}{\sqrt{(x-a)(x-b)(x-c)}} = \frac{i\pi}{M\left(\sqrt{c-a}, \sqrt{c-b}\right)},
	   \end{equation*}
	where $M(x, y)$ is the arithmetic-geometric mean of $x$ and $y$.
	   \end{prop}
	   
 \begin{thm}
	    	$E_{(n)}\cong \mathbb{C}/\left(\frac{\pi}{\sqrt{n}} G \mathbb{Z}[i]\right)$ over $\mathbb{C}$.
	    \end{thm}
\begin{proof}
	  	   Let $a=-n$, $b=0$, $c=n$ in Proposition 4.2, we can get
	
	  \begin{equation*}
	  	\omega_1 = \frac{\pi}{M\left(\sqrt{2n}, \sqrt{n}\right)} \quad \text{and} \quad \omega_2 = \frac{i\pi}{M\left(\sqrt{2n}, \sqrt{n}\right)}.
	  \end{equation*}
	
	  Since $M\left(\sqrt{2n}, \sqrt{n}\right) = \sqrt{n} M(\sqrt{2}, 1)$ and
	$$1/M(\sqrt{2}, 1) = G \quad \text{(Gauss's constant, approximately } 0.8346268 \ldots \text{)},$$	
	  we have	
	  \begin{equation*}
	  	\omega_1 = \frac{\pi}{\sqrt{n} }G \quad \text{and} \quad \omega_2 = \frac{i \pi}{\sqrt{n} }G.
	  \end{equation*}	
	  So	
	  \begin{equation*}
	  	\Lambda = \mathbb{Z} \frac{\pi}{\sqrt{n} }G + i \mathbb{Z} \frac{\pi}{\sqrt{n} }G = \frac{\pi}{\sqrt{n} }G \mathbb{Z}[i].
	  \end{equation*}
	  \end{proof}

\section{Modularity of Elliptic Curves}
	\vskip 3mm
	   \begin{prop}(Silverman, \cite{Silverman1986})
	   Let \( E \) be an elliptic curve defined over \( \mathbb{Q} \) with conductor \( N \). Then there exists a map defined over \( \mathbb{Q} \),
\[
\varPhi : X_0(N) \longrightarrow E,
\]
which is called the modularity map

	   \end{prop}

Based on the modular form of the elliptic curve $E$, it has \( L \)-series coefficient $a_2 = 0$ and $a_m = 0$ for $m = 3 \pmod{4}$, so we only need to calculate $a_m$ where $m = 1 \pmod{4}$.
	
\begin{prop}
	  For $n=5$, let $E_{(5)}$ denote the elliptic curve $y^2 = x^3 - n^2x$. The modularity map $\Phi$ of  $E_{(5)}$ is defined by	
	  $$\Phi: X_0(800) \to E_{(5)}$$ 
Then we have $\Phi$ is
	  \begin{equation*}
\begin{aligned}
	  	\Phi(q)=&q-\frac{1}{3}q^9 - \frac{6}{13}q^{13} - \frac{2}{17}q^{17} - \frac{10}{29}q^{29} + \frac{2}{37}q^{37} + \frac{10}{41}q^{41} - \frac{1}{7}q^{49} +\\
 &+\frac{14}{53}q^{53} - \frac{10}{61}q^{61} +\frac{6}{73}q^{73} + \frac{1}{9}q^{81} - \frac{18}{97}q^{97}+\dots.
\end{aligned}	  
\end{equation*}
	  In addition, if we consider $X_0(N)$ as algebraic variety $(j(\omega),j_N(\omega))$, then $\Phi$ is defined over $\mathbb{Q}$.
	  \end{prop}	

\begin{proof}
	  We can quickly conclude that the modular form of $E_{(5)}$ is
	  \begin{equation*}
\begin{aligned}
	  	f(q) =& q-3q^9 - 6q^{13} - 2q^{17} - 10q^{29} + 2q^{37} + 10q^{41} - 7q^{49} +\\
&+ 14q^{53} - 10q^{61} + 6q^{73} + 9q^{81} - 18q^{97}  + O(q^{100}).
\end{aligned}	   
\end{equation*}
      \end{proof}

\section{Heegner Points}
\vskip 3mm
\begin{definition}
$\omega$ is a Heegner point of $X_0(N)$ if it satisfies \[ A\omega^2 + B\omega + C = 0 \] where \( A, B \) and \( C \) are relatively prime integers and \( A \equiv 0 \pmod{N} \), $D = B^2 - 4AC \equiv r^2 \pmod{4N}$, $r \equiv B \pmod{2N}$.
\end{definition}

\begin{prop}(Gross and Zagier, \cite{Gross1983})
	Suppose $\epsilon = -1$, assume $U=\Phi (\omega_1) + \dots + \Phi (\omega_h)$, we choose \{$\omega_i$\} as all heegner point with the same $N$ and $D$, then $(\wp(U), \wp'(U)) \in E(\mathbb{Q})$.
\end{prop}

\begin{proof}
Because $\Phi$ is a rational map defined over $\mathbb{Q}$, so we have $\Phi(\omega) \in E(\mathbb{Q}(j(\omega),j_N(\omega)))$. From class field theory, we know that $\mathbb{Q}(j(\omega),j_N(\omega))$ is the Hilbert class field of $\mathbb{Q}(\omega)$, and $\mathbb{Q}(\omega,j(\omega))$ is a Galois extension of $\mathbb{Q}(\omega)$. We also have the equation
\[
[\mathbb{Q}(\omega,j(\omega)) : \mathbb{Q}(\omega)] = h(D),
\]
where $h(D)$ is the ideal class number.

We obtain $\Phi(\omega) \in \mathbb{Q}((\omega),j(\omega))$, since
\[
\mathbb{Q}(j(\omega),j_N(\omega)) \subset \mathbb{Q}(\omega,j(\omega))
\]

We can find Heegner points $\omega_i$, each satisfying $D = B_i^2 - 4A_iC_i$, $D \equiv r^2 \pmod{4N}$, $r \equiv B_i \pmod{2N}$. Then we obtain the sum $\Phi (\omega_i) + \dots + \Phi (\omega_h) \in E(\mathbb{Q}(\omega))$.

We now assume $U = \Phi(\omega_1) + \ldots + \Phi(\omega_n)$.

Suppose $\epsilon = -1$, then $U = \overline{U} \in \mathbb{R}$. We have $(\wp(U), \wp'(U)) \in E(\mathbb{R})$.

In a word, $(\wp(U), \wp'(U)) \in E(\mathbb{R}) \cap E(K) = E(\mathbb{Q})$.
\end{proof}

Next, we continue with $n=5$ as an example, and first solve for the points of $x^2 \equiv -31 \pmod{4 \times 800}$.

There is one of solutions $x \equiv 113 \pmod{4 \times 800}$.

We take $D = -31$, $h(D) = 3$, and thus we find three Heegner points on $X_0(800)$.

Take
\[
\begin{aligned}
	\omega_1 &= \frac{\sqrt{-31}}{16000} - \frac{3313}{16000}, \\
	\omega_2 &= \frac{\sqrt{-31}}{16000} - \frac{12687}{16000}, \\
	\omega_3 &= \frac{\sqrt{-31}}{16000} - \frac{1397}{16000}.
\end{aligned}
\]

Estimate $U = \Phi(\omega_1) + \Phi(\omega_2) + \Phi(\omega_3) \approx -0.874107405430\ldots+i1.726197864\ldots$
\begin{prop}(Silverman, \cite{Silverman1986})
	For the elliptic curve $E_{(n)}: y^2 = x^3 - n^2x$ over $\mathbb{Q}$. The torsion subgroup is $\mathbb{Z}/2\mathbb{Z} \oplus \mathbb{Z}/2\mathbb{Z}$ and  $(0,0)$, $(n,0)$ are its generator.
\end{prop}

In fact, one can prove the points $\varphi ^{-1}(0,0)$, $\varphi ^{-1}(n,0)$, $\varphi ^{-1}(-n,0)$ lie in $\mathbb{C}/\Lambda$ where $\varphi ^{-1}$ is introduced in Proposition 4.1, these respectively are $(1-i) \frac{\pi}{\sqrt{n}} G \mod \Lambda$, $\frac{1}{2} \frac{\pi}{\sqrt{n}} G \mod \Lambda$, $-\frac{i}{2} \frac{\pi}{\sqrt{n}} G \mod \Lambda$. Now we define $S_n$ as \{$0 \mod \Lambda ,(1-i)\frac{\pi}{\sqrt{n}} G \mod \Lambda$, $\frac{1}{2} \frac{\pi}{\sqrt{n}} G \mod \Lambda$, $-\frac{i}{2} \frac{\pi}{\sqrt{n}} G \mod \Lambda$ \}

\begin{thm}
	If $\epsilon = -1$, our calculation result $U = \Phi(\omega_1) + \ldots + \Phi(\omega_n)$ is not in $S_n$, then we can conclude that $n$ is a congruent number.
\end{thm}

\begin{proof}
    This is evident because \( U \) in this case is a non-torsion point, i.e., a non-trivial point. By the knowledge of congruent numbers, we know that \( n \) is a congruent number if and only if the equation \( y^2 = x^3 - n^2 x \) has a non-trivial point.
\end{proof}

Because $\frac{1}{2}\frac{\pi}{\sqrt{5}}$ which is about $0.5863098932314025360654776255\mod \Lambda$ is not in $S_5$, then we can conclude that $5$ is a congruent number.
Next, we calculate the Pythagorean triple (a,b,c)
\[
(\wp(U), \wp'(U)) \approx (11.67361, -36.048179629629)
\]
\[
(\wp(U), \wp'(U))=\left(\frac{1050625}{90000}, \frac{62279}{1728}\right)
\]

Consequently,$\left(\frac{1050625}{90000}, \frac{62279}{1728}\right)$ is a nontrivial solution of  $E_{(5)}: y^2 = x^3 - n^2x$.
\begin{thm}(Knapp, \cite{Knapp1992})
	Let $E$ be an elliptic curve over a field of characteristic
	not equal to 2 or 3. Suppose $E$ is given by
	$$y^{2}=(x-\alpha)(x-\beta)(x-\gamma)=x^{3}+rx^{2}+sx+t$$
	with $\alpha,\beta,\gamma$ in $K.$ For $(x_2,y_2)$ in $E(K)$, there exists $(x_1,y_1)$ in $E(K)$ with $2(x_1,y_1)=(x_2,y_2)$ if and only if $x_2-\alpha,x_2-\beta$ and $x_2-\gamma$ are squares
	$in k.$	
\end{thm}

According to Theorem 6.2,  we have $2(\frac{1050625}{90000}, \frac{62279}{1728})=(\frac{11183412793921}{2234116132416}, \frac{-1791076534232245919}{3339324446657665536}),$
\[
\begin{aligned}
	&a=(x+n)^{1/2}+(x-n)^{1/2}, \\
	&b=(x+n)^{1/2}-(x-n)^{1/2}, \\
	&c=2x^{1/2}.
\end{aligned}
\]

After computation, we have
	\begin{equation*}
a = \frac{4920}{1519} ,
	b = \frac{1519}{492},
	c = \frac{3344161}{747348}.
\end{equation*}

\vskip 8mm

\section{Proof of The Congruent Number 13}
\vskip 3mm
Although we already know that primes \( p \) congruent to 5 modulo 8 are congruent numbers, we will prove that 13 is congruent number by using a constructive proof as introduced earlier.

Based on the properties of the elliptic curve \( y^2 = x^3 - n^2 x \), we have the following proposition.

\begin{prop}
	The conductor of the curve \( y^2 = x^3 - 13^2 x \) is 5408, and its torsion structure is \( \mathbb{Z}/2\mathbb{Z} \oplus \mathbb{Z}/2\mathbb{Z} \).	The modular form is given by
\begin{equation*}
\begin{aligned}
	&q + 2 q^5 - 3 q^9 + 2 q^{17} - q^{25} - 10 q^{29} + 2 q^{37} - 10 q^{41} - 6 q^{45} - 7 q^{49} +\\
 &+14 q^{53} - 10 q^{61} + 6 q^{73} + 9 q^{81} + 4 q^{85} - 10 q^{89} - 18 q^{97}.
	\end{aligned}
\end{equation*}
	
	As an elliptic curve over \( \mathbb{C} \), it is isomorphic to \( \mathbb{C}/\left(\frac{\pi}{\sqrt{13}} G \mathbb{Z}[i]\right) \).
\end{prop}

\begin{thm}
	13 is a congruent number.
\end{thm}

\begin{proof}
	Let \( D = -55 \), and we know that \( h(D) = 4 \) based on the following:
	
	By factorizing \( 21632 \), we get \( 21632 = 2^7 \times 13^2 \). We need to solve \( x^2 \equiv -55 \) modulo \( 2^7 \) and modulo \( 13^2 \).
	
	1. For the modulo \( 2^7 = 128 \):
	\[
	x^2 \equiv -55 \mod{128}
	\]
	We find that \( x = 29 \) is a solution.
	
	2. For the modulo \( 13^2 = 169 \):
	\[
	x^2 \equiv -55 \mod{169}
	\]
	After calculations, we find that \( x = 72 \) is a solution.
	
	We select the following:
	\[
	z_1 = \frac{\sqrt{-55}}{10816} + \frac{8547}{10816},
	\]
	\[
	z_2 = \frac{\sqrt{-55}}{10816} + \frac{5987}{10816},
	\]
	\[
	z_3 = \frac{\sqrt{-55}}{10816} + \frac{4829}{10816},
	\]
	\[
	z_4 = \frac{\sqrt{-55}}{10816} + \frac{2269}{10816}.
	\]
	
	The value \( U = \Phi(z_1) + \Phi(z_2) + \Phi(z_3) + \Phi(z_4) \approx -2.3665268305 + 4.23177966 E-37\cdot I \mod \Lambda \) is not in $S_{13}$.
\end{proof}

In fact, the rational point we have computed, via the Weierstrass map, is:
\[
(\wp(U), \wp'(U)) \approx  (30.4381658050760, 151.843211275078),
\]
\[
(\wp(U), \wp'(U)) = \left(\frac{11432100241}{375584400}, \frac{1105240264347961}{7278825672000}\right).
\]

Therefore, \( \left( \frac{11432100241}{375584400}, \frac{1105240264347961}{7278825672000} \right) \) is a non-trivial solution of the curve \( E: y^2 = x^3 - 13^2 x \).

\vskip 8mm

\noindent{\bf Acknowledgement}.
This work is supported in part by Shaanxi Fundamental Science Research Project for Mathematics and Physics
(Grant No. 23JSY033) and BJNSF 1242013.

\end{document}